\documentclass[a4paper,11pt]{amsart}
\usepackage{amsfonts}
 \usepackage{enumerate}
 \usepackage{amsmath, amsthm, amscd, amssymb}
\usepackage{pdfsync}
\usepackage[title]{appendix}
\usepackage[all]{xy}
\usepackage{latexsym,graphicx}
\usepackage{xspace}
\usepackage{mathabx}
\usepackage{array}
\usepackage{newclude}
\usepackage{hyperref}
\usepackage{url}
\usepackage{enumitem}
\hypersetup{
  colorlinks,
  citecolor=black,
  linkcolor=black,
  urlcolor=black}
  
  \usepackage[utf8]{inputenc}
  \usepackage[T1]{fontenc}
 
\usepackage{mathrsfs}

\renewcommand*{\HyperDestNameFilter}[1]{\jobname-#1} 
\numberwithin{equation}{section}
\usepackage[centering, marginparwidth=2cm]{geometry}
\usepackage{marginnote}

\addtocontents{toc}{\setcounter{tocdepth}{1}}

 \newcommand{\Addresses}{{% additional braces for segregating \footnotesize
  \bigskip
  \footnotesize

\textsc{CNRS, IMJ-PRG, Sorbonne Universit\'{e}, 4 place Jussieu, 75005 Paris, (France)}\par\nopagebreak
  \textit{E-mail address}, G.~Baldi: \texttt{baldi@imj-prg.fr} 

  \medskip

\textsc{Hausdorﬀ Center for Mathematics, University of Bonn,
Endenicher Allee 60,
and
Max Planck Institute for Mathematics,
Vivatsgasse 7, Bonn (Germany)}\par\nopagebreak
  \textit{E-mail address}, J.~Lam: \texttt{lam@mpim-bonn.mpg.de}
}}

\usepackage{cleveref}
\usepackage{thmtools}

\usepackage{tikz-cd}
\usetikzlibrary{cd}
\usepackage{leftidx}
\usetikzlibrary{positioning}
\usepackage{dsfont}
\usepackage{tikz}
\usetikzlibrary{matrix,arrows,decorations.pathmorphing,positioning}

\newcommand\blfootnote[1]{%
\begingroup
\renewcommand\thefootnote{}\footnote{#1}%
\addtocounter{footnote}{-1}%
\endgroup
}

\theoremstyle{plain}
\newtheorem{theor}{Theorem}[section]
\newtheorem{claim}{Claim}
\newtheorem{conj}[theor]{Conjecture}
\newtheorem{lem}[theor]{Lemma}
\newtheorem{defi}[theor]{Definition}

\newtheorem{prop}[theor]{Proposition}
\newtheorem{cor}[theor]{Corollary}
\theoremstyle{definition}

\newtheorem{rmk}[theor]{Remark}
\newtheorem{question}[theor]{Question}
\theoremstyle{remark}
\numberwithin{equation}{subsection}

\crefname{theor}{Theorem}{Theorems}\Crefname{theor}{Theorem}{Theorems}
\crefname{conj}{Conjecture}{Conjectures}\Crefname{conj}{Conjecture}{Conjectures}
\crefname{lem}{Lemma}{Lemmas}\Crefname{lem}{Lemma}{Lemmas}
\crefname{prop}{Proposition}{Propositions}\Crefname{prop}{Proposition}{Propositions}
\crefname{cor}{Corollary}{Corollaries}\Crefname{cor}{Corollary}{Corollaries}
\crefname{defi}{Definition}{Definitions}\Crefname{defi}{Definition}{Definitions}
\crefname{rmk}{Remark}{Remarks}\Crefname{rmk}{Remark}{Remarks}
\crefname{example}{Example}{Examples}\Crefname{example}{Example}{Examples}
\crefname{question}{Question}{Questions}\Crefname{question}{Question}{Questions}

\newcommand{\SU}{\operatorname{SU}}

\newcommand{\ad}{\textnormal{ad}}

\newcommand{\MT}{\mathbf{MT}}

\newcommand{\Ad}{\operatorname{Ad}}

\DeclareMathOperator{\Id}{Id}
\DeclareMathOperator{\Mod}{Mod}

\DeclareMathOperator{\tr}{Tr}

\newcommand{\Aut}{\operatorname{Aut}}

\newcommand{\rk}{\operatorname{rk}}

\newcommand{\Comm}{\operatorname{Comm}}

\newcommand{\id}{\operatorname{Id}}

\newcommand{\VV}{{\mathbb V}}

\newcommand{\HL}{\textnormal{HL}}

\newcommand{\kab}{\mathrm{Kab}}

\newcommand\nc{\newcommand}
\nc\on{\operatorname}
\nc\renc{\renewcommand}

\nc\mf\mathfrak
\nc\mc\mathcal
\nc\mb\mathbb
\nc\msf\mathsf
\nc\mscr\mathscr

\newcommand{\Z}{\mathbb{Z}}
\newcommand{\Q}{\mathbb{Q}}

\newcommand{\R}{\mathbb{R}}

\newcommand{\A}{\mathbb{A}}

\newcommand{\Oo}{\mathcal{O}}

\usepackage[normalem]{ulem}

\newcommand{\C}{\mathbb{C}}

\newcommand{\Qbar}{\overline{\mathbb{Q}}}

\newcommand{\GL}{\mathbf{GL}}
\newcommand{\SL}{\mathbf{SL}}
\newcommand{\PSL}{\mathbf{PSL}}
\newcommand{\G}{{\mathbf G}}

\newcommand{\Tr}{\mathrm{Tr}}

\usepackage{microtype}

\begin{document}

\title{Murphy's Law in non-abelian Hodge theory}\date{\today}  
\author{Gregorio Baldi and Yeuk Hay Joshua Lam} 

\begin{abstract}
We construct explicit examples of non-integral variations of $\mathbb{Q}$-Hodge structures. Our approach leverages Fenchel--Nielsen-type parameterizations, due to Kabaya and Maskit, of the Teichmüller component of relative character varieties. Additionally, we discuss various Diophantine results concerning the $\mathcal{O}_{K,S}$-integral points of such character varieties, and give a new proof of Beauville's classical theorem on families of elliptic curves. We conclude by collecting \emph{pathological behaviors} of the Hodge locus of non-integral $\Q$VHS, in particular the failure of the  Cattani--Deligne--Kaplan theorem and the Andr\'e--Oort conjecture;  our results indicate  that for $\mb{Q}$VHS which are not $\mb{Z}$VHS, what can go wrong must go wrong.

\end{abstract}
\maketitle

\tableofcontents

\section{Introduction}\label{sec:intro}

Let $X$ be a smooth quasi-projective variety, and $R $ a ring included in $\R$ or equal to $\C$. Throughout this paper, by $R$VHS we mean a \emph{polarized pure} variation of Hodge structures with $R$-coefficients over $X$, in the sense of Simpson's influential paper \cite{Simpson}. 

\subsection{Motivation: Simpson's non-abelian Hodge and standard conjectures}
The Hodge conjecture provides a criterion for determining when a cohomology class in $H^*(Y,\Q)$ of a smooth projective complex variety $Y$ arises as a  $\Q$-linear combination of algebraic cycles. The \emph{non-abelian} analogue of the classical Hodge conjecture is often formulated as follows (cf. \cite[Conj. 12.4]{zbMATH01126801}):

\begin{conj}
Let $X$ be a smooth complex quasi-projective variety. A $\Z$-local system on $X$ that underlies a complex polarized variation of Hodge structure is of ``geometric origin''.
\end{conj}

One first difference is that the classical Hodge conjecture does not hold with $\Z$-coefficients, as first observed by Atiyah and Hirzebruch in \cite{zbMATH03177307}, but, nevertheless, the non-abelian version is stated with $\Z$-coefficients. Thus, a natural question is: could the non-abelian Hodge conjecture be true for $\Q$-local systems? The simple fact that the conjecture was stated with $\Z$-coefficients suggests that the answer should be no. However, what is an example of a non-integral (or that is not of ``geometric origin'') $\Q$VHS? A simple way to do so is to take a unitary character valued in  $\overline{\mb{Q}}$, and to take the sum of  its Galois conjugates; a slight generalization is to take pushforwards of such $\mb{Q}$VHS along maps of algebraic varieties. One could rule out such examples by insisting that the local systems have sufficiently large monodromy. To that end, we prove the following:
\begin{theor}\label{thm: main-direct-factor-formulation}
For any $g \geq 2$ there are infinitely many non-isomorphic genus $g$ curves supporting a rank two local system, with   trivial determinant and Zariski dense monodromy in $\SL_2$, which is a complex direct factor of a non-integral  $\Q$VHS. Moreover, the adjoint local system $\ad(\rho)$ underlies a non-integral $\mb{Q}$VHS.

\end{theor}
\begin{rmk}
    In the theorem above, the need to pass to the adjoint to get a $\mb{Q}$VHS is due to the usual problem that, for an arbitrary group, a $\mb{C}$-representation of dimension $r$  with traces in the field $\mb Q$ need not be defined on a $\mb{Q}$-vector space of dimension $r$. See \Cref{lem: trace-field-Q-implies-QVHS} for a precise formulation.
\end{rmk}
Other examples of things that ``go wrong'' in non-abelian Hodge theory are discussed, whence the title \emph{Murphy's Law}.

The $\mb{Q}$VHS in \Cref{thm: main-direct-factor-formulation} are evidently not of geometric origin, since the associated local systems are not integral. On the other hand, Simpson's standard conjecture \cite[p. 372]{simpson1990transcendental} says that $\mb{C}$-local systems which are bi-algebraic points of the Riemann--Hilbert correspondence are of geometric origin. One could optimistically make the analogous conjecture with the Riemann--Hilbert correspondence replaced by the Simpson correspondence: namely that a bi-algebraic point of the Simpson correspondence is a  local system of geometric origin; we refer to this conjecture as the \emph{Higgs standard conjecture}.

We find that the examples in \Cref{thm: main-direct-factor-formulation} furnish interesting test cases of this conjecture. Concretely, we have the following:
\begin{conj}[Consequence of Higgs standard conjecture]\label{conj: special-higgs-standard}
    Each genus $g$ curve occurring in \Cref{thm: main-direct-factor-formulation} is not defined over $\overline{\mb{Q}}$.
\end{conj}
To deduce \Cref{conj: special-higgs-standard} from the Higgs standard conjecture, the key point is that each Higgs bundle corresponding to a local system in the statement of \Cref{thm: main-direct-factor-formulation} is intrinsic to the underlying curve, and is defined over $\Qbar$ whenever the curve is.
\subsection{Further results}

Let $\mathrm{Primes} = \{2, 3, 5, 7, \dots\}$ denote the set of (positive) prime numbers. The first question we address is the following. Given a finite subset $S \subset \mathrm{Primes}$, are there pairs $(X,\VV)$ of smooth projective varieties supporting $S$-integral but not $S'$-integral for any smaller subset $S' \subset S$? How often does this phenomenon happen? To clarify the above question we need a definition. First of all, whenever $\rho$ denotes a (complex) local system (e.g. the one associated to $\mathbb{V}$), we set its trace ring to be
\begin{equation}\label{trring}
    \Z[\rho]:= \Z[\Tr(\rho(\gamma)) \mid \gamma \in \pi_1(X)].
\end{equation}

\begin{defi}
Let $S$ be a finite set of primes. Let $X$ be a smooth quasi-projective variety and $\rho$ an irreducible complex local system on $X$. We say $\rho$
\begin{itemize}
    \item \emph{has property $(P_S)$} if it has trivial determinant, trace ring \textbf{equal} to $\Z[S^{-1}]$, and it is a complex direct factor of an irreducible $\Q$VHS on $X$. 
    \item \emph{has property $(\widetilde{P}_S)$} if it has trivial determinant, trace ring \textbf{containing} $\Z[S^{-1}]$, and it is a complex direct factor of an irreducible $\Q$VHS on $X$. 
\end{itemize}

\end{defi}

\begin{theor}\label{thm: main}
For any $g \geq 2$ and any non-empty finite set $S \subset \mathrm{Primes}$, there exist infinitely many non-isomorphic, smooth, projective curves $X/\mb{C}$ of genus $g$ supporting a rank two complex local system with property $(\widetilde{P}_S)$. Moreover these rank two complex local system have discrete and faithful monodromy. 
\end{theor}

If $S=\emptyset$, then there are only finitely many genus $g$ curves supporting a $\Z$VHS as we explain in \Cref{rmkZpoints}.

\begin{rmk}
In the study of rank two local systems, the restriction to the case of curves appears quite natural by virtue of the Corlette--Simpson theorem \cite{MR2457528}. Similar questions in the case of higher rank local systems (or, more precisely, local systems with monodromy bigger than $\SL_2$), appear very difficult and we are aware of no strategies to produce non-trivial examples.
\end{rmk}

Although \Cref{thm: main} deals with projective curves, it is also possible to find examples in the case of  quasi-projective curves (with, say, infinite monodromies at the punctures, to avoid trivial examples). We say that a Riemann surface has \emph{signature $(g,n)$} if it is obtained from a compact genus $g$ Riemann surface by removing $n$ points. Our second main result deals with the signature $(0;4)$ version of \Cref{thm: main} and gives a new proof of  a classical result of Beauville \cite{zbMATH03794233} (in the case where $S=\emptyset$).

\begin{theor}\label{thm1.6}
Let $S$ be a finite set of primes.
\begin{itemize}
    \item If $S$ is empty, then there are only four curves of signature $(0;4)$  supporting a rank two local system with property\footnote{If $S=\emptyset$, property ($P_S$) simply asserts that they are direct factors of a $\Z$VHS. More generally we will prove that there are only four curves of signature $(0;4)$ that admit a modular embedding. See \Cref{4integersol} for a precise statement.} ($P_S$);
    \item If $S$ is non-empty then there are infinitely many non-isomorphic curves of signature $(0;4)$ supporting a rank two local system with property ($P_S$). Moreover these rank two complex local system have discrete and faithful monodromy.
\end{itemize} 
\end{theor}
\begin{rmk}
    For the first part of \Cref{thm1.6}, we are in fact able to plot  all the points in the Teichm\"uller space corresponding to the signature $(0; 4)$ Riemann surfaces--see \Cref{fig:funddomain} for details.
\end{rmk}
\begin{rmk}
    To prove the second part of \Cref{thm1.6}, we prove that the $\mc{O}_{K,S}$-points of the Teichm\"uller component of a signature $(0;4)$ character variety form infinitely many mapping class group orbits, as long as $(K, S)\neq (\mb Q, \emptyset)$. Related questions on integral points of character varieties have been studied by Esnault--de Jong and Coccia--Litt \cite{2025arXiv250700167C} and have applications towards the Ekedahl--Shepherd-Barron--Taylor conjecture \cite{lam-litt-non-ab-p-curv}, although here we are more interested in mapping class group orbits of $S$-integral points rather than $S$-integral points themselves--see \Cref{coccialittconj} for details.
\end{rmk}
The VHS we construct are \emph{uniformizing} in the sense of Simpson \cite[Sec. 9]{zbMATH04096412}. Specifically, we find them by considering the Teichmüller component of the character variety and certain explicit parametrizations appearing in the work of Maskit \cite{zbMATH01887454} and Kabaya \cite{kabaya}. 
The same techniques can be applied to study $K$VHS for real number fields $K$ larger than $\Q$, and, in fact, we obtain a number of \emph{Diophantine results} about $S$-integral points of character varieties. In the process we will also collect refinements of a conjecture of  Coccia--Litt \cite{2025arXiv250700167C}, at least after it is specialized to the case of one dimensional bases. We will come back to this point in \Cref{coccialittconj}.

On a related note, we remark here that the literature contains two definitions of $K$VHS: the one given above and another that requires the local systems associated with every embedding of $K$ into $\C$ to underlie a VHS (see e.g. \cite[Def. 5.1]{zbMATH08084572}). We also construct examples showing that these two definitions do not generally agree, as one might expect.

\begin{theor}\label{another thm}
Let $K \subset \R$ be a totally real number field. There exist non-unitary $K$VHS (both integral and non-integral) that are not $K$-factors of a $\Q$VHS. More precisely, for some embedding $K \to \R$ different from the identity, the resulting $\R$-local system is not unitary and does not underlie a VHS (but it is faithful and discrete). 
\end{theor}

\begin{rmk}
Certain related examples were also found by S. Ricker, P. Schmutz Schaller, and J. Wolfart \cite{zbMATH01805844, zbMATH01463511}; see also \Cref{rmkquadri}.
\end{rmk}

Finally, with the aim of collecting certain \emph{Murphy's laws in non-abelian Hodge theory}, in \Cref{sec3} we show that the Hodge locus of a non-integral $\mb Q$VHS behaves very differently from that of a $\Z$VHS: for example, the analogue of Andr\'e--Oort fails, in that the CM points can be dense without the $\mb Q$VHS being special (in the sense of unlikely intersection conjectures).

\subsection{Plan of the paper}
\Cref{prelim} fixes the relevant notions of $R$VHS as well as the notion of relative character varieties and some basics of Teichmüller theory. \Cref{sectioneasy} treats the unitary case. %as well as the special case of quadrilateral groups. 

\Cref{sec:04} develops Maskit's explicit parametrization of the Teichm\"uller component in signature $(0;4)$ and gives a new proof of Beauville's theorem. In the following \Cref{sec:dichotomy} we study the $S$-integral Markoff-type solutions. \Cref{sec:kabaya} recalls Kabaya's Fenchel--Nielsen parametrization, which is used in \Cref{sec:proofmain} to prove \Cref{thm: main} in arbitrary genus.

 \Cref{sec3} collects more examples of Murphy's law for the Hodge locus of non-integral $\Q$VHS with discrete monodromy.

\subsection*{Acknowledgments}

G.B. was partially supported by the grant ANR-HoLoDiRibey of the Agence Nationale de la Recherche. During the final stages of this work he was at the Institute for Advanced Studies of
Princeton and he is grateful for the amazing working conditions. He also thanks the Ambrose Monell Foundation as well as the Giorgio and Elena Petronio Fellowship Fund for
support.

J. L. was supported partially by the
DFG Walter Benjamin grant Dynamics on character varieties and Hodge theory (project number: 541272769). He is also grateful to the MPIM Bonn for excellent working conditions. 

\section{Preliminaries}\label{prelim}

\subsection{Trace field and VHS}\label{sec:def}

In this section, we explain how a $\R$VHS with trace field $\Q$ gives rise to a $\Q$VHS. This point is subtle, but well understood. For more on this topic, we refer to \cite{zbMATH04158959}, the discussion in \cite[Rmk. 2.2.3]{2025arXiv250202147F}, as well as \cite{zbMATH05560295}, and especially the end of Sec. 4 in \cite{Simpson}.

Let $G$ be a real semisimple algebraic group without compact factors and $\Gamma$ be a finitely generated subgroup of $G$. Denote by
\begin{displaymath}
\Ad: \Gamma \subset G \xrightarrow{\Ad} \Aut(\mathfrak{g})
\end{displaymath}
the adjoint representation in the automorphisms of the Lie algebra $\mathfrak{g}$ of $G$. For more regarding the following definition, see the so called Bass-Serre theory and in particular Bass' paper \cite{MR586867}.
\begin{defi}
We define the \emph{adjoint trace field} of $\Gamma$ as the field generated over $\Q$ by the set
\begin{displaymath}
\{\Tr \Ad (\gamma) : \gamma \in \Gamma\},
\end{displaymath}
and the \emph{adjoint trace ring} by 
\begin{displaymath}
\Z[\{\Tr \Ad (\gamma) : \gamma \in \Gamma\}].
\end{displaymath}
\end{defi}

The adjoint trace field is a finitely generated field extension of $\Q$. Let $G$ be the real group given by the Zariski closure of the monodromy representation, and assume that it is semisimple. Then, by a theorem of Vinberg \cite{MR0279206}, there is a $K$-form of $G$ such that the image of the monodromy representation lies in its $K$-points (up to finite index and real conjugation). In the case $G=\SL_2(\R)^+$, we also have the \emph{trace field}, namely the field generated by traces of elements of $\Gamma$; it is the fraction field of the trace ring appearing in \eqref{trring}, and contains the adjoint trace field.
%is a subfield of the trace field appearing in \eqref{trring}.

We conclude this subsection by recording two useful results.
\begin{lem}\label{lem: trace-field-Q-implies-QVHS}
    Let $X$ be a smooth complex variety, and $\rho: \pi_1(X)\to \SL_2(\mb{C})$ a representation which has Zariski dense image and is furthermore non-unitary (i.e. cannot be conjugated into $\SU_2$). Suppose that $\rho$ underlies a $\mb{C}$VHS, and that the trace field of $\rho$ is $\mb{Q}$. Then $\rho$ is a complex direct factor of a $\mb{Q}$VHS. 

    Moreover, if the image of $\rho$ contains a non-trivial unipotent element, then $\rho$ is itself a $\mb{Q}$VHS.
\end{lem}
For a similar argument see also the bottom of page 56 and page 57 of \cite{Simpson} (in particular Lemma 4.8 from \emph{op. cit.} which is credited to M. Larsen).
\begin{proof}
    First,  that the trace field of $\rho$ is $\mb{Q}$ implies that there is a quaternion algebra $B/\mb{Q}$ and a representation $\rho_B: \pi_1(X)\rightarrow B^{\times}$ such that the composition 
    \[
\pi_1(X)\xrightarrow{\rho_B}B^{\times} \to (B \otimes_{\mb{Q}} \mb{C})^{\times} \cong\SL_2(\mb{C})
    \]
    is isomorphic to $\rho$; here the last isomorphism is given by a splitting of $B$ over $\mb{C}$, i.e. an isomorphism $B\otimes \mb{C}\cong M_2(\mb{C})$. Indeed, the existence of  $B$ and $\rho_B$  follows immediately from \cite[Theorem 3.2.1]{MaclachlanReid} and the construction therein of the quaternion algebra; note that to apply Theorem 3.2.1 of \emph{loc. cit.} we need to check that $\rho(\pi_1(X))$ is \emph{non-elementary}, and this follows from our assumption that $\rho$ has Zariski dense image and is non-unitary--see \cite[Proof of Lemma 2.2.5]{2025arXiv250202147F} where this argument is recorded.

    Now, as is well-known, $B$ splits over an imaginary quadratic extension $F/\mb{Q}$, and hence $\rho$ factors as 
\[
\pi_1(X)\xrightarrow{\rho_F}\SL_2(F)\to \SL_2(\mb{C})
    \]
    for a (in fact, either) choice of  embedding of $F\to \mb{C}$. Then $\rho_F$ underlies a $\mb{Q}$VHS of rank four, and $\rho_F\otimes_{\mb Q}\mb{C}=\rho\oplus \overline{\rho}$, as required.
    %(where on $\overline{\rho}$ we take the complex conjugate VHS of that on $\rho$), as required.

    For the last claim in the lemma statement, the existence of a non-trivial unipotent implies that $B$ is already split over $\mb{Q}$, so there is no need to pass to the extension $F$.
\end{proof}

\begin{rmk}
    In the statement of \Cref{lem: trace-field-Q-implies-QVHS}, if one weakens the assumption to that the trace field is a  number field $K$ (instead of $\mb Q$), the same argument shows  that $\rho$ is a complex direct factor of a $K$VHS.
\end{rmk}
\begin{cor}\label{cor: adjoint-tr-field-Q-implies-VHS}
    Let $Y$ be a smooth complex variety and $\rho: \pi_1(Y)\to \SL_2(\mb{C})$ a representation with Zariski dense and non-unitary image. Suppose that the adjoint trace field of $\rho$ is $\mb{Q}$, and that $\rho$ underlies a $\mb{C}$VHS. Then $\rho$ is a complex direct factor of a $\mb Q$VHS. 
\end{cor}
\begin{proof}
    Let $\Gamma = \rho(\pi_1(Y))\subset \SL_2(\mb{C})$, and consider the subgroup $\Gamma^{(2)}:=\langle \gamma^2: \gamma\in \Gamma\rangle$ of $\Gamma$.

    By \cite[Lem. 3.3.3]{MaclachlanReid}, $\Gamma^{(2)}$ has finite index in $\Gamma$, and hence corresponds to a finite \'etale cover $\pi: Y'\to Y$. Note that the trace field of $\Gamma^{(2)}$ is contained in the adjoint trace field of $\Gamma$: this follows from \cite[Lem. 3.5.6]{MaclachlanReid}, and see also the last sentence of \cite[Proof of Lem. 2.2.5]{2025arXiv250202147F} where this simple argument is recorded. 

    The upshot is that, under our assumption that the adjoint trace field of $\Gamma$ is $\mb{Q}$, we may  apply \Cref{lem: trace-field-Q-implies-QVHS} to deduce that $\pi^*\rho$ is a complex direct factor of a $\mb{Q}$VHS, say $\mb{V}$, on $Y'$. Therefore, $\rho$ is a complex direct factor of $\pi_*\mb{V}$, as required.
\end{proof}

Finally, we record, without proof, the following standard statement:
\begin{prop}\label{propnoproof}
    If $\Gamma$ is a group, and $n$ is a positive integer,  every complex direct factor of a representation $\Gamma\to \GL_n(\mb{Z})$ has traces in $\overline{\mb Z}$.
\end{prop}

\subsection{(Relative) Character varieties and the Teichm\"uller component}\label{teichsec}

Let $\Sigma=\Sigma_{g,n}$ be a connected oriented surface of genus $g$ with
$n$ punctures, and assume $\chi(\Sigma)<0$.  When $n>0$, choose simple loops
$\delta_1,\ldots,\delta_n$ around the punctures, with the usual convention that
they are well-defined up to conjugacy in $\pi_1(\Sigma)$.  For
$G=\SL_2$ or $\PSL_2$, the $G$-character variety of $\Sigma$ is the affine GIT
quotient
\[
  X_G(\Sigma)
  :=
  \operatorname{Hom}(\pi_1(\Sigma),G)// G,
\]
where $G$ acts by conjugation.  In the $\SL_2$ case, fixing boundary traces
$c=(c_1,\ldots,c_n)\in \C^n$ gives the relative character variety
\[
  X_{\SL_2}(\Sigma,c)
  :=
  \{\rho\in X_{\SL_2}(\Sigma):\tr\rho(\delta_i)=c_i
    \text{ for all }i\}.
\]
Equivalently, it is the fiber of the boundary-trace morphism
$X_{\SL_2}(\Sigma)\to\A^n$.  In the $\PSL_2$ case one fixes the conjugacy
classes of the peripheral monodromies instead. After
choosing $\SL_2$ lifts, this corresponds to boundary traces equal to $\pm2$.

The mapping class group $\Mod(\Sigma)$ acts on $\pi_1(\Sigma)$ by outer
automorphisms and hence acts algebraically on the character variety by
precomposition.  If mapping classes are required to preserve the punctures, or
to preserve them individually, then the action preserves the corresponding
relative character variety.

A marked complete hyperbolic structure on $\Sigma$ has holonomy
\[
  \rho:\pi_1(\Sigma)\longrightarrow \PSL_2(\R)\subset \PSL_2(\C),
\]
which is discrete, faithful, and type-preserving: puncture loops are sent to
parabolic elements.  Thus the Teichm\"uller space $\mathcal{T}_{g,n}$ embeds into the real points of the
relative character variety (see e.g. \cite[Part 2]{zbMATH05960418}). We call its image the \emph{Teichm\"uller component}. If one works with $\SL_2$ rather
than $\PSL_2$, a choice of lifts of the peripheral elements selects a lift of
this component; changing lifts changes some trace coordinates by signs, but
does not change the trace ring generated by those coordinates.

Let $K\subset\R$ be a real number field and let $S$ be a finite set of finite
places of $K$, fix a signature $(g;n)$ and boundary data as above.

\begin{defi}\label{teichpoints}
 The set of $\Mod (\Sigma)$ orbits of $\Oo_{K,S}$-points of the relative character variety lying on the Teichm\"uller component is called the \emph{set of $\Oo_{K,S}$-Teichm\"uller points}.
\end{defi}
This is the bridge between the Hodge-theoretic problem and the Diophantine
problem studied below.  A point of the Teichm\"uller component gives the
uniformizing local system of the corresponding Riemann surface, hence a complex
variation of Hodge structure.  The trace ring of this local system is generated
by the traces of the corresponding character.  Therefore constructing
$S$-integral, but not $S'$-integral, uniformizing local systems amounts to
constructing $\mathcal O_{K,S}$-Teichm\"uller points with the prescribed trace
ring.

\begin{rmk}\label{sec:examples} For the four-punctured sphere with all boundary traces equal to $2$, take trace
coordinates
\[
  X=\tr(AB),\qquad Y=\tr(BC),\qquad Z=\tr(CA),
\]
where $A,B,C,D$ are the four boundary monodromies and $ABCD=1$.  The relative
character variety is the Markoff-type cubic
\begin{equation}
\label{eq:V2}
  V_2:\quad
  X^2+Y^2+Z^2+XYZ-8X-8Y-8Z+28=0.
\end{equation}
This equation is the specialization of the so called \emph{Fricke cubic} to
$a=b=c=d=2$. See for example \cite[Eq. 3.2]{zbMATH01310582} and \cite[Sec. 5.2]{zbMATH05560295} for more details. In these coordinates the $\PSL_2$-Teichm\"uller component is the image of the connected component
\[
  \mathcal T_{0,4}
  =
  \{(X,Y,Z)\in V_2(\R): X<-2,\;Y<-2,\;Z<-2\}.
\]
\end{rmk}

\begin{rmk}
The Maskit parametrization we will use in Section~4 gives a rational parametrization of the semi-algebraic component appearing in the previous remark.  The rationality of the cubic
\eqref{eq:V2} as an algebraic surface is weaker than this statement: it
parametrizes the whole surface, whereas the Teichm\"uller component is selected by
real inequalities.
\end{rmk}

\section{Unitary and non-integral local systems}\label{sectioneasy}\label{unitary}

It is easy to construct non-integral unitary $\Q$VHS, as we now explain; the non-unitary case, which occupies the rest of the paper, is considerably more delicate. We first record the following.

\begin{prop}\label{prop: weil-number}
For any prime number  $p$, there exists an algebraic number $\tau$, living inside a quadratic number field $K$, such that 
\begin{itemize}
\item $|\sigma(\tau)|=1$ for each embedding $\sigma: K \hookrightarrow \mb{C}$, and 
\item $\tau$ is not a $p$-integer.
\end{itemize}
\end{prop}
The proof is a standard application of the Weil conjectures and the existence of ordinary elliptic curves; one can also write down directly the explicit Weil numbers involved.
\begin{proof}
Let $p$ be a prime number and $E/\mb{F}_p$ an ordinary elliptic curve. Let $\ell$ be an auxiliary prime $\neq p$. Let $u$ be the  eigenvalue of Frobenius acting on $H^1(E\otimes \overline{\mb{F}}_p, \mb{Q}_{\ell})$, and such that $u$ is a $p$-unit: such a $u$ exists by ordinarity. Notice that $u$ is independent of the choice of $\ell$. Then $u$ is a Weil number of weight one: that is, $|\sigma(u)|=p^{1/2}$ for all embeddings $\sigma: \mb{Q}(u)\hookrightarrow \mb{C}$. It then suffices to take $\tau = u^2/p$.

\end{proof}

 Let $X$ be a smooth complex quasi-projective variety which admits  a surjection $\pi: \pi_1(X)\rightarrow \mb{Z}$ (e.g. a curve). For any $y\in \mb{C}^{\times}$, let $\gamma_y: \mb{Z} \rightarrow \mb{C}^{\times}$ be the representation sending the generator $1\in \mb{Z}$ to $y$, and write $\rho_y:= \gamma_y \circ\pi: \pi_1(X)\rightarrow \mb{C}^{\times}$.  

Finally, take a prime number $p>2$ and let $\tau$ be an algebraic number output by \Cref{prop: weil-number} (for the prime $p$), and consider the representation of $\pi_1(X)$ given by 
\[
R := \rho_{\tau}\oplus \rho_{\tau'},
\]
where $\tau'$ is the Galois conjugate of $\tau$.

\begin{prop}
The rank two local system $R$ has a descent which is a $\mb{Q}$-local system, which is unitary and non-integral.  
\end{prop}
\begin{proof}
The first claim follows from the fact that the traces of $R$ lie in $\mb Q$. Now, Bb construction, $R$ is a sum of rank one unitary representations, and hence unitary; it remains to show that $R$ is not integral, i.e. that its traces do not lie in $\mb{Z}$. Since $\pi: \pi_1(X)\rightarrow \mb{Z}$ is surjective, it suffices to check that the powers of the matrix

\[
\begin{pmatrix}
\tau & \\ 
& \tau'
\end{pmatrix}
\]
do not all have traces in $\mb{Z}$. Suppose the contrary, so that $\tau+\tau', \tau^2+\tau'^2 \in \mb{Z}$. Since $p>2$, this implies that $\tau+\tau'$ and $\tau \tau'$ are $p$-integral, and therefore that $\tau$ is $p$-integral, since $\tau$ is a root of the polynomial $X^2-(\tau+\tau')X+\tau\tau'$; this contradicts the $p$-non-integrality promised by \Cref{prop: weil-number}.
\end{proof}

\section{Character variety of a 4-punctured sphere: $\Z$-points}\label{sec:04}

In this section, we begin treating the most explicit case of the paper: the $\mathcal{O}_{K,S}$-integral points of the Teichm\"uller component of the character variety of a four-punctured sphere (i.e. the case of signature $(0;4)$), together with their orbits under the mapping class group. That is, with the terminology introduced in \Cref{teichpoints}, the $\Oo_{K,S}$-Teichm\"uller points. After recalling Maskit's parametrization we identify the trace ring with three explicit coordinate, give a new proof of Beauville's classification of the four modular curves of signature $(0;4)$, produce infinitely many $S$-integral solutions of the associated Markoff-type equation.

\subsection{Maskit's parametrization and fundamental domain}\label{funddomainandeq}

Every Fuchsian group of signature $(0;4)$, acting on the upper half-plane $\mathbb{H}$, can be generated by four parabolic transformations $A, B, C, D\in \operatorname{PGL}(2,\R)^+$ with $ABCD = 1$. As in \cite{zbMATH04075492}, we normalize so that $AB$ has its attracting fixed point at $\infty$ and its repelling fixed point at $0$, and so that the fixed point of $C$ is at $1$. A set of generators with the above property is called a \emph{good set of generators}. Let $x$ be the fixed point of $D$ and $y$ the fixed point of $B$; then $x > 1$ and $y<0$. In particular the four punctures carry parabolic (unipotent) local monodromy, and the pair $(x,y)$ records the positions of the two remaining fixed points.

\subsubsection{Parametrization}
Maskit shows that $x$
and $y$ serve as parameters for the deformation space of these groups:

\begin{theor}[{\cite[Thm. 4]{zbMATH04075492}}]\label{maskitunif}
Let $x > 1$, and $y < 0$, be real numbers. Then there
is a unique (and explicit) set of good generators $A, B, C, D$, for a Fuchsian group of
signature $(0;4)$, where $ABCD = 1$, so that $D$ has its fixed point at $x$, $C$
has its fixed point at $1$, and $B$ has its fixed point at $y$.
\end{theor}

 Explicitly, given $x,y$ as above, we can take $A, B, C, D \in \operatorname{PGL(2,\R)}^+$ to be the images in $\operatorname{PGL(2,\R)}^+$ of the following matrices (see \cite[p. 254, 256]{zbMATH04075492}):

\[
A =
\begin{pmatrix}
2 x^{2} y & - x^{2} y^{2}(1+x) \\
1 + x & -2 x y
\end{pmatrix},
\qquad
B =
\begin{pmatrix}
-2 x y & y^{2}(1+x) \\
-1 - x & 2 y
\end{pmatrix},
\]

\[
C =
\begin{pmatrix}
2 & -1 - x \\
1 + x & -2x
\end{pmatrix},
\qquad
D =
\begin{pmatrix}
-2x & x^{2}(1+x) \\
-1 - x & 2 x^{2}
\end{pmatrix}.
\]

For our forthcoming applications we wish to work with lifts  $\tilde{A}, \tilde{B}, \tilde{C}, \tilde{D}$ (from $\operatorname{PGL(2,\R)}^+$ to  $\SL_2(\R)^+)$ of  a good set of generators. So we consider each matrix divided by an appropriate choice of square root of its determinant 
\begin{equation}\label{eqn: normalized-(0,4)-matrices}
  \tilde{A}:= A /  xy(x-1), \  \tilde{B}:= B /y(1-x), \  \tilde{C}:= C /(1-x) ,  \ \tilde{D}:= D / x(x-1).  
\end{equation}
It is straightforward to check that $\tilde{A}, \tilde{B}, \tilde{C}, \tilde{D} \in \SL_2(\mb{R})^{+},$ and that $\tilde{A}\tilde{B}\tilde{C}\tilde{D}=\Id.$

\begin{rmk}
These lifts are important since we are looking for a Riemann surface $X=\Gamma\backslash \mathbb{H}$ with $\Gamma \subset \operatorname{SL}_2(\R)^+$ such that its trace ring $\Z[\tr \gamma : \gamma \in \Gamma]$ has certain $S$-integral properties; cf. \Cref{sec:def}. Moreover the choice of lifts is determined by the following condition we wish to enforce: $\Tr\tilde A=\Tr\tilde B=\Tr\tilde C=\Tr\tilde D=2$.
\end{rmk}

Write $(E, F, G) = (\tilde{A}\tilde{B},\ \tilde{A}\tilde{C},\ \tilde{A}\tilde{D})$; a direct computation shows:
\begin{align*}
	\Tr E &= -\frac{x^2+1}{x}  \\
	\Tr F &= -\frac{8 x^2 y - (1+x)^2 \left(x^2 y^2 + 1\right)}{x \, y \, (x-1)^2}\\
	\Tr G&=  \frac{(1+x)^2 \left(1 + y^2\right) - 8 x y}{y \, (x-1)^2}
	\end{align*}

\subsubsection{Fundamental domain}\label{funddomain}
In the second part of the paper, Maskit \cite[Sec. 9-14]{zbMATH04075492} introduces a fundamental domain $\mathcal{D} $ for the action of the Teichm\"{u}ller modular group of signature $(0;4)$. It is given by the following region in the fourth quadrant of $\R^2$, cf. eq. (16), (20), and (21) in \emph{op. cit.}:

\begin{equation}
\begin{aligned}
y &\le -\frac{(x^2 - 4x + 1) - (x - 1)(x^2 - 6x + 1)^{1/2}}{2x}, \\
\text{or} \qquad
y &\ge -\frac{(x^2 - 4x + 1) + (x - 1)(x^2 - 6x + 1)^{1/2}}{2x}.
\end{aligned}
\end{equation}
\begin{equation}
\frac{1}{\sqrt{x}} \le|y| \le \sqrt{x}.
\end{equation}
The interior of the region so described is displayed in \Cref{fig:funddomainold}.
\begin{figure}[ht]
\centering
\includegraphics[width=0.72\textwidth]{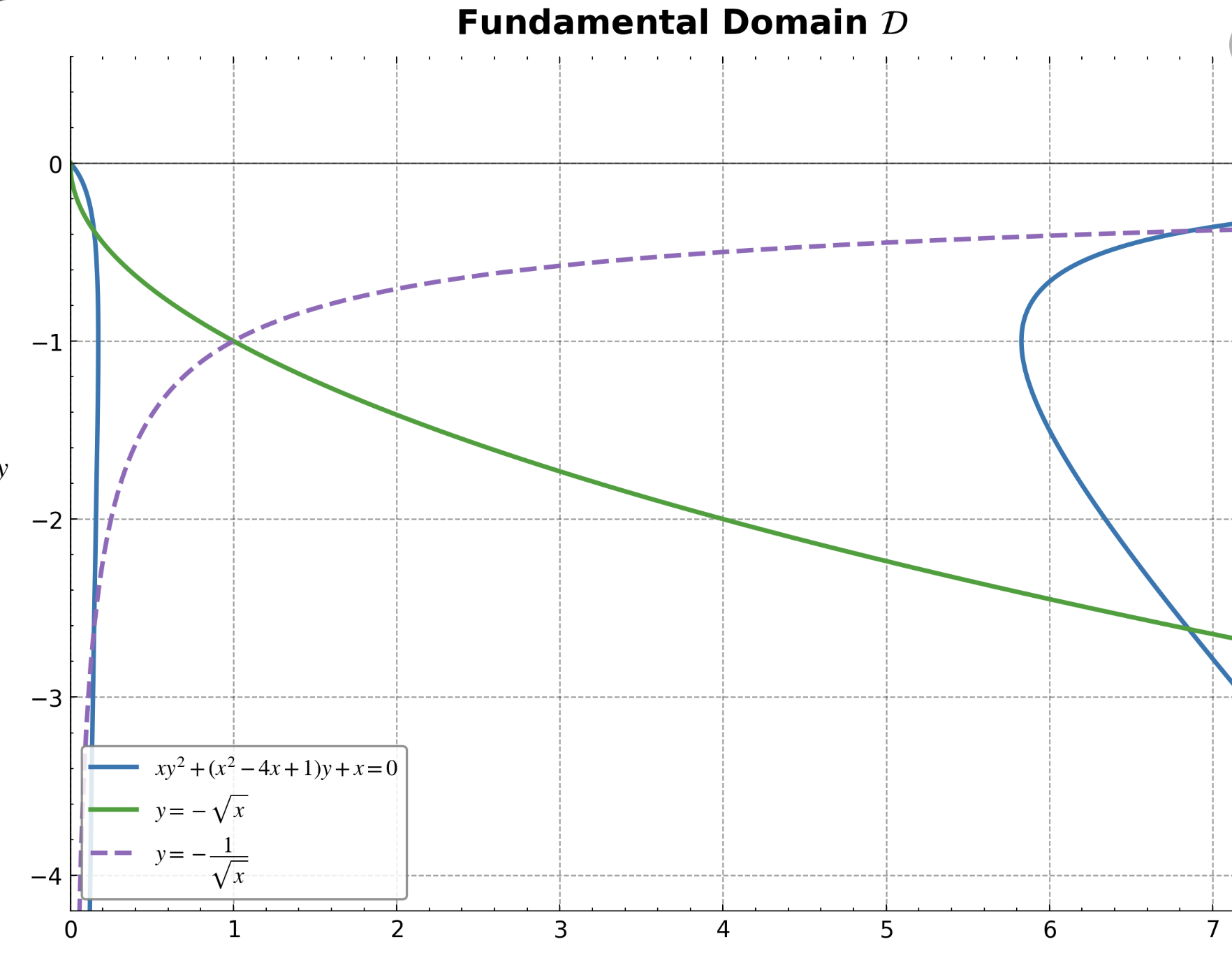}
\caption{The fundamental domain $\mathcal{D}$ --- the interior of the region bounded by the conic $xy^2+(x^2-4x+1)y+x=0$ and the curves $y=-\sqrt{x}$, $y=-1/\sqrt{x}$.}
\label{fig:funddomainold}
\end{figure}

 \begin{rmk}
 The Teichm\"uller group of signature $(0;4)$ is generated by $\alpha$ and $\beta$, where $\alpha (A,B,C,D)= (B,C,D,A)$ has order two, and $\beta(A,B,C,D)= (ABA^{-1},A,C,D) $ has infinite order. In the $\mathcal{D}$-parametrization, they act as described in \cite[page 264]{zbMATH04075492}.
\end{rmk}

We are interested in the values of the following functions at $(x,y)\in \mathcal{D}$, derived from the traces of $(E,F,G)$ from the previous paragraph: 

\begin{flalign*}
&F_1(x) = -\frac{x^2+1}{x}  \\
&F_2(x,y) = -\frac{8 x^2 y - (1+x)^2 \left(x^2 y^2 + 1\right)}{x \, y \, (x-1)^2}\\
&F_3(x,y) =  \frac{(1+x)^2 \left(1 + y^2\right) - 8 x y}{y \, (x-1)^2}
\end{flalign*}

Setting as auxiliary variable
\[
\lambda=x+\tfrac1x,
\]
and using the relations $(x-1)^2=x(\lambda-2)$ and $(1+x)^2=x(\lambda+2)$, we immediately obtain:

\begin{equation}\label{eq:Flambda}
F_1=-\lambda, \qquad F_2=\frac{(\lambda+2)(xy+\tfrac1{xy})-8}{\lambda-2},\qquad
F_3=\frac{(\lambda+2)(y+\tfrac1y)-8}{\lambda-2}.
\end{equation}

\subsection{The trace ring}\label{sec:tracering}

Let $\Gamma$ be a Fuchsian group of signature $(0;4)$ and $X=\Gamma \backslash \mathbb{H}$ be the associated hyperbolic Riemann surface. By definition (cf. \Cref{sec:def}), the integrality of the uniformizing local system attached to $X$ is measured by its trace ring. The next lemma shows that in signature $(0;4)$ this ring is generated by the three coordinates $F_1,F_2,F_3$ associated to $X$ via Maskit's \Cref{maskitunif}. In particular the $\Oo_{K,S}$-Teichm\"uller points correspond to pairs $(x,y)\in \mathcal{D}$ such that $F_1,F_2,F_3 \in \Z[S^{-1}]$ (which in turn give rise to $\Q$VHS via \Cref{lem: trace-field-Q-implies-QVHS}). 

\begin{lem}\label{lem:tracering}
Let $\Gamma=\langle \tilde A,\tilde B,\tilde C,\tilde D\rangle\subset \SL_2(\R)^+$ be generated by the lifts $\tilde A,\tilde B,\tilde C,\tilde D$ as in \eqref{eqn: normalized-(0,4)-matrices}. Then
\[
\Z[\Tr(\gamma):\gamma\in\Gamma]\;=\;\Z[\Tr E,\Tr F,\Tr G],
\]
where, as we recall,   $(E, F, G) = (\tilde{A}\tilde{B},\ \tilde{A}\tilde{C},\ \tilde{A}\tilde{D})$.
\end{lem}

\begin{proof}

The fundamental group of the four-punctured sphere is freely generated by
$\tilde A,\tilde B,\tilde C$; note that
\[
  \tilde D=(\tilde A\tilde B\tilde C)^{-1}.
\]
We use the integral Fricke--Vogt trace-generation theorem for a free group of
rank three: the $\SL_2$ trace ring of the representation is generated over
$\Z$ by
\[
\Tr\tilde A,\quad \Tr\tilde B,\quad \Tr\tilde C,\quad
\Tr(\tilde A\tilde B),\quad
\Tr(\tilde B\tilde C),\quad
\Tr(\tilde A\tilde C),\quad
\Tr(\tilde A\tilde B\tilde C).
\]
This is the specialization of Goldman's \cite[Sec. 5.1]{zbMATH05560295} description
of the rank three character ring; the eighth trace
$\Tr(\tilde A\tilde C\tilde B)$ is eliminated by the sum relation therein.

We have 
\[
  \Tr\tilde A=\Tr\tilde B=\Tr\tilde C=\Tr\tilde D=2,
\]
and 
\[
  \Tr(\tilde A\tilde B\tilde C)
  =
  \Tr(\tilde D^{-1})
  =
  \Tr(\tilde D)
  =
  2.
\]
The only generators of the trace ring left are therefore
\[
  \Tr(\tilde A\tilde B),\qquad
  \Tr(\tilde A\tilde C),\qquad
  \Tr(\tilde B\tilde C).
\]
By definition $E=\tilde A\tilde B$ and $F=\tilde A\tilde C$. Moreover,
\[
  \Tr G
  =
  \Tr(\tilde A\tilde D)
  =
  \Tr\!\left(\tilde A\tilde C^{-1}\tilde B^{-1}\tilde A^{-1}\right)
  =
  \Tr(\tilde C^{-1}\tilde B^{-1})
  =
  \Tr(\tilde B\tilde C),
\]
where we used cyclicity of trace and $\Tr(M^{-1})=\Tr(M)$ for
$M\in\SL_2$.  Thus the full $\SL_2$ trace ring is
\[
  \Z[\Tr(\gamma):\gamma\in\Gamma]
  =
  \Z[\Tr E,\Tr F,\Tr G].
\]

\end{proof}

\subsection{A new take on Beauville's theorem}\label{sec:beauville}

A smooth hyperbolic curve over $\mb{C}$ (or Riemann surface) is said to admit a  \emph{modular embedding} if its uniformizing local system (which by convention is rank two and has trivial determinant) is of geometric origin--we refer the reader to \cite[Defn 2.3.1]{litt2024motives} for the definition of a local system of geometric origin, as well as further properties and questions concerning them. In this particular setting, it simply means that the local system is the first cohomology of a family of elliptic curves. We recover  the following theorem of Beauville's \cite{zbMATH03794233}. 

\begin{theor} \label{4integersol}
    There are only four  Riemann surfaces of signature $(0;4)$ admitting modular embeddings. They correspond to the following points of $\mathcal{D}$ (in the $x,y$ coordinates introduced in \Cref{funddomain}):
    \begin{displaymath}
         P_1= \left(\frac{3+\sqrt{5}}{2}, -1\right), P_2= \left(2+\sqrt{3}, -1\right),
    \end{displaymath}
    \begin{displaymath}
         P_3=\left(3+\sqrt{8}, -1\right), P_4= \left(\frac{7+\sqrt{45}}{2}, -\sqrt{\frac{7+\sqrt{45}}{2}}\right) .
    \end{displaymath}
\end{theor}

\begin{rmk}
    We have plotted the four signature $(0;4)$ Riemann surfaces in Teichm\"uller space in \Cref{fig:funddomain}.
\end{rmk}

\begin{figure}[ht]
\centering
\includegraphics[width=0.72\textwidth]{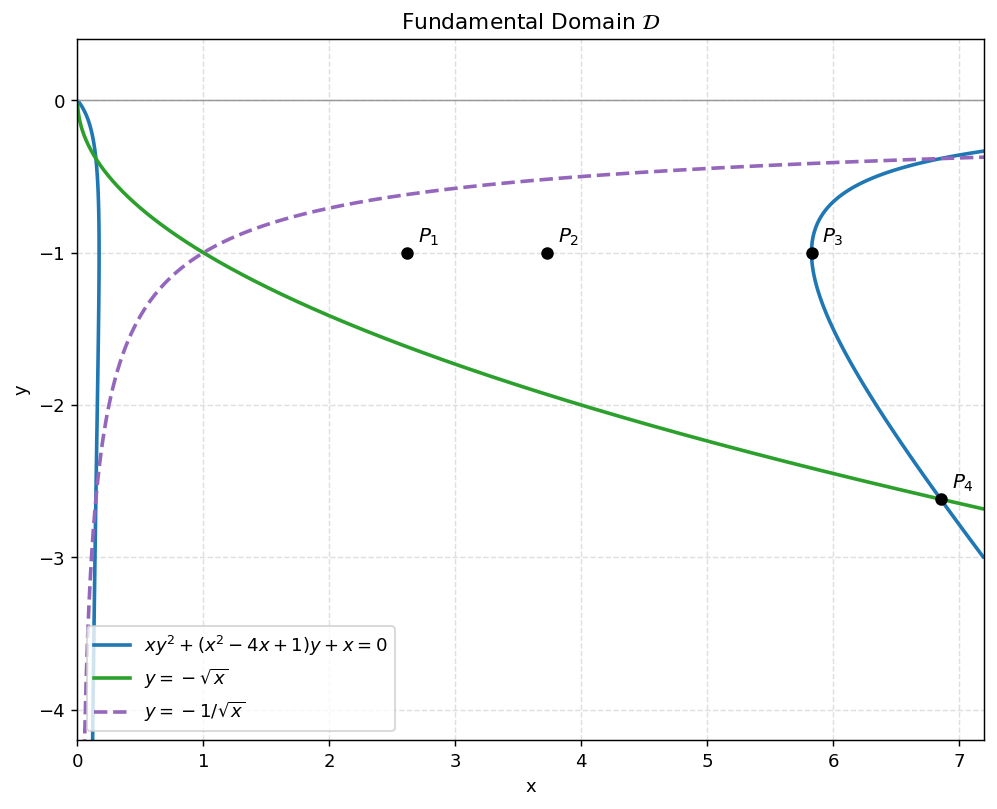}
\caption{The four modular points $P_1,\dots,P_4$ of \Cref{4integersol}.}
\label{fig:funddomain}
\end{figure}

\begin{proof}
Let $X=\mb{P}^1\setminus D$ be a Riemann surface of signature $(0; 4)$. Suppose that $X$ admits a modular embedding, and let $\mb{V}$ be the rank two, trivial determinant, uniformizing local  system on $X$; our assumption is that $\mb{V}$ is of geometric origin. 
\begin{claim}
    $\mb{V}$ has trace ring $\mb Z$.
\end{claim}
\begin{proof}[Proof of Claim]
For $\gamma\in \pi_1(X)$, let $\Tr(\gamma)$ denote its trace in the representation of $\pi_1(X)$ given by $\mb{V}$.  Since $\Tr(\gamma)\in \overline{\Z}$ for all $\gamma\in \pi_1(X)$, to prove the claim it suffices to show that $\Tr(\gamma)\in \mb{Q}$. Let $\sigma$ be an arbitrary automorphism of $\mb{C}$, and let $\mb{V}^{\sigma}$ be $\sigma$-conjugate local system of $\mb{V}$: that is $\mb{V}^{\sigma}:= \mb{V}\otimes_{\mb{C}, \sigma}\mb{C}$. It suffices now to prove that $\mb{V}\cong \mb{V}^{\sigma}$. Note that $\mb{V}^{\sigma}$ is also of geometric origin, and so corresponds under the Corlette--Simpson correspondence to a graded, stable, parabolic Higgs bundle on $(\mb{P}^1, D)$; moreover, the parabolic structure at each point in $D$ is trivial since the local monodromies are unipotent. It is now straightforward to check that the graded Higgs bundle corresponding to $\mb{V}^{\sigma}$ must be 
\[
(\mathscr{O}(1)\oplus\mathscr{O}(-1), \theta),
\]
where $\theta$ is the graded Higgs field induced by an isomorphism $\mscr{O}(1)\rightarrow \mscr{O}(-1)\otimes \Omega^1_{\mb{P}^1}(\mathrm{log} D)$. But all such graded parabolic Higgs bundles are isomorphic, and hence $\mb{V}\cong \mb{V}^{\sigma}$, which proves the claim. This argument is recorded in \cite[Proof of Theorem 1.1]{lam2022motivic} and we refer the reader there for more details.
\end{proof}

By \Cref{lem:tracering} we are left to  classify the $(x,y)\in\mathcal{D}$ with $F_1(x,y),F_2(x,y),F_3(x,y)\in\Z$.
We argue using the change of variables described in \eqref{eq:Flambda}: integrality of $F_1$ forces $\lambda=x+ \frac{1}{x}\in\Z$; since $x>1$ we have $\lambda>2$, hence $\lambda\ge3$. We claim that $\lambda\le7$. Indeed, for $(x,y)\in\mathcal{D}$ we have $y\in[-\sqrt x,-1/\sqrt x]$. On this interval the function $y\mapsto u=y+\tfrac1y$ attains its maximum $-2$ at $y=-1$ and its minimum $-(\sqrt x+1/\sqrt x)=-\sqrt{\lambda+2}$ at the two endpoints; thus $u\in[-\sqrt{\lambda+2},-2]$. Moreover the conic boundary condition $xy^2+(x^2-4x+1)y+x\ge0$ defining $\mathcal{D}$ is, after dividing by $-y>0$ and then by $x>0$, equivalent to $-u\ge\lambda-4$, i.e. $u\le 4-\lambda$. Hence a point of $\mathcal{D}$ with the given $\lambda$ exists only if $-\sqrt{\lambda+2}\le\min(-2,\,4-\lambda)$. For $\lambda\ge6$ this reads $\lambda-4\le\sqrt{\lambda+2}$, i.e. $(\lambda-2)(\lambda-7)\le0$, so $\lambda\le7$. Therefore we have to consider integers $\lambda$ such that $3\le\lambda\le7$.

To conclude we perform the finite check. For each $\lambda\in\{3,4,5,6,7\}$ the value of $x$ is determined, and by \eqref{eq:Flambda} the condition $F_3\in\Z$ pins $u$ --- hence $y$ --- to finitely many values in $[-\sqrt{\lambda+2},\min(-2,4-\lambda)]$, after which $F_2\in\Z$ is one further (finite) condition. Carrying this out yields exactly the four points $P_1\ (\lambda=3)$, $P_2\ (\lambda=4)$, $P_3\ (\lambda=6)$ and $P_4\ (\lambda=7)$ of the statement, and no solution for $\lambda=5$; the corresponding trace triples are recorded in \Cref{table: beauville-comparison}, and the points are displayed in \Cref{fig:funddomain}.
\end{proof}

\subsubsection{Comparison with Beauville's approach}\label{sec:beauville-comparison}

Beauville \cite{zbMATH03794233} computed the complete list of stable families of elliptic curves over the projective line which have exactly four singular fibers. More precisely he found 6 isogeny classes and 4 families up to isomorphism and gave equations of the families as well as the number of connected components of the singular fibers.

We can explicitly check that the local systems we obtained in \Cref{4integersol} are the same as those on Beauville's list. To do so, we recall the notation for various congruence subgroups of $\SL_2(\Z)$ (following the notation of \emph{op. cit.}):
\begin{displaymath}
    \Gamma(n)=\left\{
\begin{pmatrix} a & b \\ c & d \end{pmatrix}\in \SL_2(\mathbb{Z})
\;\middle|\; b \equiv c \equiv 0,\; a \equiv 1 \pmod{n}
\right\},
\end{displaymath}
\begin{displaymath}
    \Gamma_0^{0}(n)=\left\{
\begin{pmatrix} a & b \\ c & d \end{pmatrix}\in \SL_2(\mathbb{Z})
\;\middle|\; c \equiv 0,\; a \equiv 1 \pmod{n}
\right\},
\end{displaymath}
\begin{displaymath}
    \Gamma_{0}(n)=\left\{
\begin{pmatrix} a & b \\ c & d \end{pmatrix}\in \SL_2(\mathbb{Z})
\;\middle|\; c \equiv 0 \pmod{n}
\right\}.
\end{displaymath}

\begin{theor}[Beauville]\label{beauv}
    The only subgroups of $\SL_2(\mathbb Z)$ of signature $(0; 4)$ are 
    \begin{displaymath}
    \Gamma^0_0(5),\  \Gamma^0_0(6),\  \Gamma_0^0(4)\cap \Gamma(2),\  \Gamma_0(8)\cap \Gamma_0^0(4), \  \Gamma(3),  \ \Gamma_0(9)\cap \Gamma_0^0(3).
    \end{displaymath}
\end{theor}

\begin{table}[h]
\centering
\begin{tabular}{|c|c|c|c}
\hline
$(x,y)$-coordinates & $(\Tr E, \Tr F, \Tr G)$ & Beauville family   \\ 
\hline
$P_1=(\frac{3+\sqrt{5}}{2}, -1)$ & $(-3, -23, -18)$ & $\Gamma^0_0(5)$   \\
$P_2=(2+\sqrt{3}, -1)$ &  $(-4, -16, -10)$ & $\Gamma^0_0(6)$ \\ 
$P_3=(3+\sqrt{8}, -1)$ & $(-6, -14, -6)$ & $\Gamma_0^0(4)\cap \Gamma(2)$, $\Gamma_0(8)\cap \Gamma_0^0(4)$  \\ 
$P_4=(\frac{7+\sqrt{45}}{2}, -\sqrt{\frac{7+\sqrt{45}}{2}})$ & $(-7, -34, -7)$& $\Gamma(3), \Gamma_0(9)\cap \Gamma_0^0(3)$  \\ 
\hline
\end{tabular}

\caption{Matching between our list and Beauville's}
\label{table: beauville-comparison}
\end{table}

\begin{prop}
    The list of \Cref{beauv} corresponds to the one of \Cref{4integersol}, as displayed in \Cref{table: beauville-comparison}.
\end{prop}

\begin{proof}
     Elements in $\Gamma^0_0(5)$ have traces which are $2 
\ \bmod 5$; the only tuple $(\Tr E, \Tr F, \Tr G)$ with this property is $(-3, -23, -18)$, and this gives the matching in the first row. In fact, computing the traces (modulo a suitable prime) of elements of $\Gamma_0^0(6), \Gamma_0^0(4) \cap \Gamma(2)$, and $\Gamma(3)$ allows us to match each item in Beauville's list to one on our list.
\end{proof}

\begin{rmk}
We can be even more explicit for $P_3=(x, y)=(3+\sqrt{8}, -1)$; for completeness we record here the argument. We must check that the  local system specified by this choice of $(x,y)$ agrees with the local system for $\Gamma_0(8)\cap \Gamma_0^0(4)$, which is on Beauville's list. To check that these two local systems are isomorphic, it suffices to compare the traces of $(E, F, G)$ to the corresponding matrices of the $\Gamma_0(8)\cap \Gamma_0^0(4)$-local system; we will now do so. One can check that the local system on $\mb{P}^1\setminus \{0, \pm 1, \infty\}$ given by the group $\Gamma_0(8)\cap \Gamma_0^0(4)$ has monodromy tuple (see \cite[Eqn (C.1)]{collas2018monodromy})
\[
(M_{-1},\, M_{1},\, M_{0},\, M_{\infty})
=
\left(
\begin{pmatrix}
1 & 0\\
2 & 1
\end{pmatrix},
\begin{pmatrix}
-19 & -8\\
50 & 21
\end{pmatrix},
\begin{pmatrix}
-7 & -4\\
16 & 9
\end{pmatrix},
\begin{pmatrix}
-3 & -4\\
4 & 5
\end{pmatrix}
\right).
\]
One checks straightforwardly that 
\[
\operatorname{tr}(M_{-1}M_{0}, M_{-1}M_{1},  M_{-1}M_{\infty})=(-6, -14, -6),\]
agreeing with the traces given in \autoref{table: beauville-comparison}.
\end{rmk}

\section{Character variety of a 4-punctured sphere: $\mathcal{O}_{K,S}$-points}\label{sec:dichotomy}

 We now explore whether the set of $\Oo_{K,S}$-Teichm\"uller points of the relative character variety $V_2$ introduced in \Cref{sec:examples} is finite or not. We have already seen that there are only four $\Z$-Teichm\"uller points, up to the action of the mapping class group.

 In our Teichm\"uller component there is explicit mechanism to construct $\Oo_{K,S}$-points (but more care is needed to control their ring of definition): they can be constructed from an $S$-unit of infinite order. Recall that, by Dirichlet's $S$-unit theorem, $\rk\mathcal{O}_{K,S}^\times=r_1+r_2+|S|-1$, where $r_1$ and $r_2$ are the numbers of real and complex places of $K$. Since all fields we consider have a real embedding, $\rk\mathcal{O}_{K,S}^\times\ge1$ as long as $(K,S)\neq (\Q,\emptyset)$. The first main result of the section is the following:

\begin{theor}\label{conj:dichotomy}
Let $K\subset\R$ be a number field and $S$ a finite set of finite places of $K$. The set of $\Oo_{K,S}$-Teichm\"uller points with trace field exactly $K$ is finite up to the mapping class group action  if and only if $(K,S)=(\mb Q, \emptyset)$.

\end{theor}
Thanks to \Cref{lem:tracering}, the above theorem can be equivalently stated as follows: the set $(x,y)\in \mc{D}$ such that  the trace triples $(F_1,F_2,F_3)\in\mathcal{O}_{K,S}^3$, whose generated field is $K$,  is finite if and only if $(K,S)=(\mb Q, \emptyset)$. We start with a simple proposition that summarizes the previous computations and the relations of Maskit:
\begin{prop}\label{prop:unitfamily}
Suppose $\rk\mathcal{O}_{K,S}^\times\ge1$ and let $\epsilon \in\mathcal{O}_{K,S}^\times$ be a unit of infinite order with $\epsilon>1$ in the real place $K\subset\R$. For $b\ge1$ put  $\lambda_b=2+\epsilon^{-b}$, and $x_b$ to be the bigger (and necessarily real) solution of $x_b+1/x_b=\lambda_b$. Then $(x,y)= (x_b,-1)\in\mathcal{D}$ and $F_1=F_1(x,y), F_2=F_2(x,y), F_3=F_3(x,y)$ are given by 
\[
\{F_1=-2-\epsilon^{-b},\qquad F_2=-16\,\epsilon^{\,b}-6-\epsilon^{-b}, \qquad  F_3=-16\,\epsilon^{\,b}-2 \} \subset \mathcal{O}_{K,S}.
\]
Indeed the $\lambda_b$ are pairwise distinct, and the corresponding points lie in distinct orbits. Hence there are infinitely many $\Oo_{K,S}$-Teichm\"uller points. Moreover their trace rings and trace fields are
\[
\Z[F_1,F_2,F_3]=\Z[\epsilon^{-b},16\epsilon^b],
\qquad
\Q(F_1,F_2,F_3)=\Q(\epsilon^{-b}).
\]
\end{prop}
\begin{proof}
First of all we claim that the points $(x_b,-1)$ lie in $\mathcal{D}$. Indeed here $|y|=1$, so the condition $1/\sqrt{x}\le|y|\le\sqrt x$ reads $x_b\ge1$, which holds; and the value of the conic at $y=-1$ is $-x_b^2+6x_b-1$, which is $\ge0$ iff $x_b\in[3-2\sqrt2,3+2\sqrt2]$, i.e. iff $\lambda_b\le6$ --- true since $\lambda_b<3$.  As $\lambda_b\in(2,3)$ (since $\epsilon >1$ and $b > 0$) we have $(x_b,-1)\in\mathcal{D}$, and $\epsilon^{\pm b}\in\mathcal{O}_{K,S}$ because $\epsilon$ is a unit. Distinct $b$ give distinct points of $\mathcal{D}$, hence distinct orbits. The computation of the trace ring (and field) follows from the following:
\[
\epsilon^{-b}=-F_1-2,\qquad 16\epsilon^b=-F_3-2,\qquad F_2=F_1+F_3-2.
\]

\end{proof}

\begin{proof}[Proof of \Cref{conj:dichotomy}]

If $(K,S)=(\Q, \emptyset)$, then \Cref{4integersol} gives exactly four orbits. Conversely suppose $\rk\mathcal{O}_{K,S}^\times\ge1$. We first choose an $S$-unit generating the field. Indeed, $K$ has only finitely many proper subfields. For every proper subfield $F\subsetneq K$, let $S_F$ be the set of finite places of $F$ below the places in $S$. Then
\[
\mathcal{O}_{K,S}^\times\cap F^\times\subset \mathcal{O}_{F,S_F}^\times .
\]
Since $K$ has a real embedding, every proper subfield has strictly fewer archimedean places than $K$, and $|S_F|\le |S|$. Dirichlet's theorem therefore gives\footnote{Here we used that, for a non-trivial extension of number fields $F\subset K$,  $(r_1+r_2)(F)\leq (r_1+r_2)(K)$  with equality if and only if it is a CM extension. Indeed, if the degree of $K/F$ is $n$, then each pair of complex embeddings of $F$ contributes one to $r_2(K)$, while each real embedding $v$ of $F$ contributes $n-b_v$ where $b_v$ is the number of complex embeddings of $K$ extending $v$. Therefore $r_1+r_2$ is increasing with equality if and only if $n=2, b_v=1$, i.e. $K/F$ is CM.}
\[
\rk \mathcal{O}_{F,S_F}^\times
=r_1(F)+r_2(F)+|S_F|-1
<
r_1(K)+r_2(K)+|S|-1
=\rk\mathcal{O}_{K,S}^\times .
\]

A finite union of lower-rank subgroups cannot cover the positive-rank abelian group $\mathcal{O}_{K,S}^\times$, so there is an $S$-unit $\epsilon$ with $\Q(\epsilon)=K$. Replacing $\epsilon$ by $-\epsilon$ or $\epsilon^{-1}$ if needed, we may also assume $\epsilon>1$ in the chosen embedding $K\subset\R$.

Apply \Cref{prop:unitfamily} to this $\epsilon$. The trace field of the $b$th point is $\Q(\epsilon^b)$. For infinitely many $b$ we have $\Q(\epsilon^b)=K$: indeed, in a normal closure, if an automorphism $\sigma$ fixes $\epsilon^b$, then $\sigma(\epsilon)/\epsilon$ is a $b$-th root of unity; choosing $b$ coprime to the finite group of roots of unity forces the stabilizer of $\epsilon^b$ to equal the stabilizer of $\epsilon$. Restricting to these infinitely many $b$ gives infinitely many distinct mapping-class-group orbits with trace field exactly $K$.

\end{proof}

\begin{question}\label{conj:exact-dichotomy}
Is the statement of \Cref{conj:dichotomy} true with `trace field $K$' replaced by `trace ring $\mathcal{O}_{K,S}$'? What happens when we replace the four-punctured $\mb{P}^1$ by an $n$-punctured genus $g$ Riemann surface, or $\SL_2$ by other reductive groups $G$?
\end{question}
\begin{rmk}\label{rmkZpoints}
    In the case of $\mb{Z}$-points and the $\SL_2$-relative character variety, it is easy to see that one always has finitely many mapping class group orbits on the Teichm\"uller component. Indeed, a $\mb{Z}$-point on the  Teichm\"uller component of the character variety of a signature $(g; n; e_i)$ surface\footnote{By signature $(g; n; e_1, \cdots, e_k)$ we mean a genus $g$ surface with $n$ cusps and elliptic points of order $e_1, \cdots, e_k$.} corresponds to a signature $(g; n)$ Riemann surface $Y$, whose uniformizing local system underlies a $\mb{Z}$VHS. This implies that $Y$ is a finite \'etale cover of the modular curve $Y(1)$. On the other hand, it follows from Riemann--Hurwitz that there are only finitely many such curves of a given signature. 
\end{rmk}
\begin{rmk} \label{coccialittconj}
Recall that in \cite{2025arXiv250700167C}, Coccia--Litt conjectured that $\overline{\mb{Z}}$-points are dense on relative character varieties of a smooth complex variety $Y$ equipped with a smooth normal crossings compactification, and furthermore proved it in the case of $\SL_2$-character varieties. \cref{conj:dichotomy} and \cref{conj:exact-dichotomy} are different in a number of ways: the most crucial difference is that we are asking about finitude or infinitude of \emph{mapping class group orbits} of $\mc{O}_K$-points, whereas \cite{2025arXiv250700167C} only asks about the $\mc{O}_K$-points themselves. Of course, infinitude of mapping class group orbits of $\mc{O}_K$ (or $\overline{\mb{Z}}$) points implies infinitude of the $\mc{O}_K$ points themselves, but it is not clear to us whether the converse should hold--indeed, in \emph{op.cit.} the infinitude is proven by taking the mapping class group orbit of a single point. 

\end{rmk}

\noindent The proof of \Cref{conj:dichotomy} gives the following partial result towards \Cref{conj:exact-dichotomy}:

\begin{prop}\label{prop: inf-many-Z-S-points}
Let $S$ be a finite set of primes. The number of mapping-class-group orbits of the $\Z[S^{-1}]$-solutions whose ring of definition is $\Z[S^{-1}]$ is infinite, as soon as $S\neq \emptyset$.
\end{prop}
\begin{proof}
    
For $K=\Q$ and $S\ne\emptyset$, put $\epsilon=\prod_{p\in S}p$. Then for every $b\ge1$,
\[
\Z[\epsilon^{-b},16\epsilon^b]=\Z[\epsilon^{-b}]=\Z[p^{-1}:p\in S]=\Z[S^{-1}],
\]
and the trace field is automatically $\Q$. This gives infinitely many exact trace-ring orbits. The case $S=\emptyset$ is the finite exact trace-ring case already classified in \Cref{4integersol}.
\end{proof}

\begin{proof}[Proof of \Cref{thm1.6}]
    The case $S=\emptyset$ is \Cref{4integersol}. The case $S\neq \emptyset$ follows immediately from \Cref{prop: inf-many-Z-S-points}--indeed, each mapping class group orbit in the proposition gives rise to a point on the signature $(0;4)$ Teichm\"uller space, corresponding to a signature $(0;4)$ Riemann surface  whose uniformizing local system has property $(P_S)$, by the last part of \Cref{lem: trace-field-Q-implies-QVHS}.

\end{proof}

\section{Kabaya's parametrization and proof of \cref{thm: main}}\label{sec:kabaya}

\subsection*{Notations and preliminaries}
For a matrix $A\in \SL_2(\mb{C})$, an \emph{eigenvalue parameter} is a choice of one of the eigenvalues of $A$. Given an eigenvalue parameter $e$ of $A$, a \emph{fixed point parameter} is the point $x := [v] \in \mb{P}^1(\mb{C})$, where $v$ is an eigenvector for the eigenvalue $e$. We will usually refer to the pair $(e,x)$ of an eigenvalue and fixed point parameter.

\subsection{Pants decompositions}
We call a three-holed sphere a \emph{pair of pants}. Given a genus $g$ surface $\Sigma$, a \emph{pants decomposition} is a disjoint union
\[
C=c_1\cup c_2\cup \cdots \cup c_{3g-3}
\]
of simple closed curves on $\Sigma$ such that each connected component of $\Sigma\setminus C$ is a pair of pants. We refer to a component of $C$ as an \emph{interior pants curve}. Boundary curves will not appear in the proof of \Cref{thm: main}, since there we deal only with smooth projective curves. First of all, we recall a result of Kabaya:

\begin{prop}[{\cite[Prop.~3.1]{kabaya}}]\label{prop: pop-reconstruct}
Let $P$ be a pair of pants with boundary curves $c_1,c_2,c_3$. Let $\star\in P$ be a basepoint, and for $i=1,2,3$, let $\gamma_i$ be a based loop at $\star$ encircling $c_i$ counterclockwise, so that $\gamma_1\gamma_2\gamma_3=1$. Suppose $\rho:\pi_1(P,\star)\to\SL_2(\mb{C})$ is an irreducible representation, and that each $\rho(\gamma_i)$ has two fixed points; for each $i$, fix a choice of eigenvalue and fixed point parameter $(e_i,x_i)$ for $\rho(\gamma_i)$.

Then $\rho$ may be reconstructed rationally from $e_1,e_2,e_3,x_1,x_2,x_3$: that is, $\rho$ is conjugate to a representation $\rho':\pi_1(P,\star)\to\SL_2(\mb{C})$ such that each matrix $\rho'(\gamma_i)$ has entries in $\mb{Q}(e_1,e_2,e_3,x_1,x_2,x_3)$.
\end{prop}

\subsection{Kabaya's FN coordinates}
Let $\Sigma$ be a genus $g$ surface and let $C=(c_1,c_2,\ldots,c_{3g-3})$ be a pair of pants decomposition of $\Sigma$. Define $E(\Sigma,C)\subset\mb{C}^{3g-3}$ by
\[
E(\Sigma,C):=
\left\{(e_1,\dots,e_{3g-3})\in\mb{C}^{3g-3}\ \Bigg|\ 
\begin{aligned}
& e_i\notin \{0,\pm 1\}, \quad\text{and}\\
& e_i^{\pm}e_j^{\pm}e_k^{\pm}\neq 1 \quad\text{for }\{i,j,k\}\in\mathcal{P}.
\end{aligned}
\right\}.
\]
Here $\mathcal{P}$ denotes the set of triples $\{i,j,k\}$ such that $c_i,c_j,c_k$ form the boundary of a pair of pants in $\Sigma\setminus C$.

\begin{theor}\label{thm: k-param}
There is a surjective map
\begin{equation}\label{eqn: k-param}
\kab:E(\Sigma,C)\times(\mb{C}^{\times})^{3g-3}\longrightarrow X'_{\PSL_2}(\Sigma,C).
\end{equation}
Here $X'_{\PSL_2}(\Sigma,C)$ denotes the subset of $X_{\PSL_2}(\Sigma)$ consisting of representations $\rho$ such that
\begin{enumerate}
\item for $\gamma_i\in\pi_1(\Sigma)$ in the free homotopy class of $c_i$, the action of $\rho(\gamma_i)$ on $\mb{P}^1(\mb{C})$ has two distinct fixed points,
\item the restriction of $\rho$ to each pair of pants is irreducible, and
\item $\rho$ may be lifted to a representation valued in $\SL_2(\mb{C})$.
\end{enumerate}
Moreover, let $\rho\in X'_{\PSL_2}(\Sigma,C)$ be any point, and suppose
\[
(e_1,\ldots,e_{3g-3},t_1,\ldots,t_{3g-3})\in\kab^{-1}(\rho).
\]
For $i=1,\ldots,3g-3$, let $\gamma_i$ be the based loop chosen by Kabaya in the free homotopy class of $c_i$. Then for any lift $\tilde{\rho}:\pi_1(\Sigma)\to\SL_2(\mb C)$ of $\rho$, the matrix $\tilde{\rho}(\gamma_i)$ has either $ e_i$ or $-e_i$ as an eigenvalue.

Finally, for any subfield $K\subset\mb{C}$, the image under $\kab$ of a $K$-point is again a $K$-point, i.e. corresponds to a representation into $\PSL_2(K)$.
\end{theor}
We refer to \eqref{eqn: k-param} as \emph{Kabaya's parametrization}.

\begin{proof}
The surjectivity and the eigenvalue assertion are \cite[Thm.~7.1]{kabaya}, together with the discussion preceding it. We now recall how the final $K$-rationality statement follows from Kabaya's construction.

Given a pants decomposition $C$ of $\Sigma$, Kabaya defines generators
\[
\alpha_1,\ldots,\alpha_{2g},\ \beta_1,\ldots,\beta_g
\]
of $\pi_1(\Sigma)$ in \cite[\S4.2]{kabaya}. Kabaya then writes down explicit matrices $$\tilde{\rho}(\alpha_1),\cdots, \tilde{\rho}(\alpha_{2g}), \tilde{\rho}(\beta_1), \cdots, \tilde{\rho}(\beta_g)$$ satisfying the surface-group relations, and therefore obtains a representation $\tilde{\rho}:\pi_1(\Sigma)\to\SL_2(\mb{C})$.

For the $\alpha_i$'s, the matrices are constructed in the first paragraph of \cite[p.26]{kabaya}. This is done by cutting the genus $g$ surface along simple curves to obtain a genus  zero surface, which  then has a pants decomposition for which the dual graph is a tree $\mc T$; one then constructs the $\tilde{\rho}(\alpha_i)$'s inductively as follows. First consider a leaf $v$ of $\mc T$; a representation of the $\pi_1$ of the corresponding pair of pants is determined by eigenvalue parameters and fixed point parameters (denoted by $x_i$ in \cite[p.~26]{kabaya}), and indeed the representation has entries which are rational functions of these parameters by \Cref{prop: pop-reconstruct}. Among these parameters, the former are already given to us as some of the $e_i$'s, and one may choose the latter arbitrarily. Then, for each  vertex $v'$ adjacent to  $v$, the fixed point parameters are rational functions (with $\mb{Q}$-coefficients) of the $x_i$'s for $v$, and the $e_i, t_i$'s--see \cite[Thm.~6.1]{kabaya} for the explicit formulae; with these data one  constructs a representation of $\pi_1$ of the pair of pants corresponding to $v'$, which may be glued to the existing one on $v$, and so on. 

Therefore, if one chooses the $x_i$'s for $v$ to be also in $K$, all the matrices constructed in this algorithm have entries in $K$.

For the $\beta_i$'s, Kabaya describes $\tilde{\rho}(\beta_i)$ as a lift of the unique element of $\PSL_2(\mb{C})$ sending one ordered triple of fixed point parameters $(x'_{i,1},x'_{i,2},x'_{i,3})$ to another ordered triple $(x_{i,1},x_{i,2},x_{i,3})$; we refer the reader to the second paragraph of \cite[p.26]{kabaya} for the precise choice of these triples. Since these fixed point parameters lie in  $K$, the corresponding element of $\PSL_2(\mb{C})$ is also defined over $K$. Hence the resulting point of the $\PSL_2$-character variety is a $K$-point.
\end{proof}

\subsection{Proof of \cref{thm: main}}\label{sec:proofmain}
Recall that, as explained in \Cref{teichsec}, the $\PSL_2(\mb{C})$-character variety of $\Sigma$ is, roughly, the set of all representations of its fundamental group into $\PSL_2(\mb{C})$ up to conjugacy. The space of marked hyperbolic structures on $\Sigma$ is called the Teichm\"uller space of $\Sigma$. Since a marked hyperbolic structure induces a discrete faithful representation of $\pi_1(\Sigma)$ into $\PSL_2(\mb{R})\subset\PSL_2(\mb{C})$, Teichm\"uller space can be regarded as a subspace of the $\PSL_2(\mb{C})$-character variety. 
For the proof of \cref{thm: main}, we focus on the Teichm\"uller component, since each point on it represents a VHS on a (varying) Riemann surface. Kabaya's parametrization gives explicit coordinates on this component.

\begin{lem}[Kabaya {\cite[\S12.2]{kabaya}}]\label{lem:kabaya-teich}
Under Kabaya's parametrization $\kab$, the image of the subset
\[
E_{\mathrm{Teich}}:=\{(e_1,\ldots,e_{3g-3},t_1,\ldots,t_{3g-3})\in\mb{R}^{6g-6}\mid e_i<-1,\ t_i>0\}
\]
lies in the Teichm\"uller component of $X'_{\PSL_2}(\Sigma)$.
\end{lem}

\begin{proof}[Proof of \Cref{thm: main}]
 Let $S\subset\mathrm{Primes}$ be the finite set of primes in the statement of the theorem. Recall that $S\neq \emptyset$, since otherwise there are only finitely many $\Z$-points in the Teichm\"uller component, cf. \Cref{rmkZpoints}. We must find infinitely many non-isomorphic genus $g$ Riemann surfaces, each supporting a rank two local system, which is a complex direct factor of a $\mb{Q}$VHS satisfying the required integrality and non-integrality conditions.

We do this by finding many points in genus $g$ Teichm\"uller space and taking their uniformizing local systems; by construction, these local systems underlie VHS. Define the subset $E_{\mathrm{Teich}}^{S}$ of the set $ E_{\mathrm{Teich}}$ (introduced in \Cref{lem:kabaya-teich}) by the conditions
\begin{displaymath}
    e_1=-\frac{1}{\prod_{p\in S}p}-1, \qquad e_i,t_i\in\mb{Q} \ \ \ \forall i.
\end{displaymath}

Let $X$ be the Riemann surface corresponding to $\kab(P)$ for some $P\in E_{\mathrm{Teich}}^{S}$, and let $\rho$ be the uniformizing $\PSL_2$-local system on $X$. Note that the trace field of $\rho$ is $\mb{Q}$ by construction, using the last part of \Cref{thm: k-param}. 

Let $\gamma$ be the based loop freely homotopic to $c_1$. Again by \Cref{thm: k-param}, there is a lift of $\rho$, say $\tilde{\rho}: \pi_1(\Sigma)\to \SL_2(\C)$, which moreover has either $e_1$ or $-e_1$ as an eigenvalue;  the trace ring of $\tilde{\rho}$ therefore contains $\mb{Z}[S^{-1}]$. By its construction and  \Cref{cor: adjoint-tr-field-Q-implies-VHS}, $\tilde{\rho}$ satisfies property $(\widetilde{P}_S)$, as required.
\end{proof}

\begin{proof}[Proof of \Cref{thm: main-direct-factor-formulation}]
This follows directly from the above and \Cref{propnoproof}.
    
\end{proof}

\begin{rmk}
Related parametrizations can also be found in \cite{zbMATH01887454} and \cite{zbMATH00108661}. Kabaya's parametrization is the most efficient one for our purposes because it gives rational control of the eigenvalue and twist parameters.
\end{rmk}
\begin{rmk}\label{rmkquadri}
    It is also possible to construct examples coming from \emph{quadrilateral groups}. For the relevant background, see \cite{zbMATH01463511, zbMATH01805844, zbMATH03536341}. Indeed, using  {\cite[Lem. 2 and Cor. 2]{zbMATH01463511}}, it is possible to find infinitely many quadrilateral groups of signature $[2,2,2, t]$, for $t$ an integer, corresponding to $\mb{Q}$VHS which are not integral.
\end{rmk}

\subsection{Proof of \Cref{another thm}}

 \Cref{another thm} follows from the next more general result:

\begin{theor}
Let $K\ne\Q$ be a totally real number field.  Then there exist integral
$K$-variations of Hodge structure on smooth projective curves which are not
$K$-factors of any $\Q$-variation of Hodge structure.  After allowing a finite
set of primes $S$, there also exist non-integral examples with trace ring
contained in $\mathcal O_{K,S}$ and not contained in $\mathcal O_K$.
\end{theor} 
In fact, for the examples in our proof, for any non-identity embedding, the resulting real local system does not underlie a VHS, but nevertheless has discrete and faithful monodromy.
\begin{proof}
We use the parametrization of the Teichm\"uller component from \cref{lem:kabaya-teich}. We first claim that we can  find $(e_1, \cdots, e_{3g-3}, t_1, \cdots, t_{3g-3})\in K^{6g-6}$ such that for all embeddings $\tau: K \hookrightarrow \mb{R}$, and all $i$, 
\[
\tau(e_i)<-1, \tau(t_i)>0.
\]

We first show how such a choice gives a local system on a smooth projective curve satisfying the assumptions of the theorem. Indeed, $(e_1, \cdots, e_{3g-3}, t_1, \cdots, t_{3g-3})$ corresponds to a smooth projective genus $g$ curve $Y$, whose uniformizing VHS $\mb{V}$ is $\mc{O}_K$-integral. For all other embeddings $\tau$, $\tau(\mb{V})$ stays in the Teichm\"uller component by the assumption on $(e_i, t_i)$; on the other hand, there is a unique VHS on the Teichm\"uller component of $Y$ (that is, once we fix the complex structure $Y$)--see indeed \cite[Thm 11.2]{zbMATH04032554}, which is written in the equivalent language of Higgs bundles and the $\C^*$-action.

For the $e_i$'s, we  can simply take $e_i\in \mc{O}_K$ generating $K$ and subtract a sufficiently large integer from them; similarly, for the $t_i$'s, we take $t_i\in \mc{O}_K$ and add a sufficiently large integer to them. The non-integral case is obtained analogously. 
\end{proof}

\section{Final considerations around the Hodge locus of non-integral $\Q$VHS}\label{sec3}

We continue the theme of \emph{Murphy's Law} by showing that various phenomena that hold true for $\Z$VHS, like some of the conjectures on the distribution of the Hodge locus appearing in \cite{BKU}, do not generalize to non-integral (pure and polarized) $\Q$VHS. We hope that these considerations highlight the role of integrality in the theory.  Let $X$ be a smooth quasi-projective complex variety.

First of all, recall that there is a notion of  \emph{Mumford--Tate group} of a $\Q$-Hodge structure and of a variation of $\Q$-Hodge structures. If $f: Y\to X$ is a smooth projective morphism one considers the $\Q$VHS $\VV$ naturally enriching the $\Q$-local system $R^*f_*\Q$ and  indeed wishes to compare the Mumford-Tate group at a fiber $x\in X(\C)$ with the one associated to $\VV$. However, it is crucial for the theory that, in this case, the $\Q$-local system \emph{admits an  integral structure} (given by $R^*f_*\Z$).  Nevertheless, given a $\Q$VHS, we have a generic MT group $\MT^0(\VV)$ (given by the Tannakian formalism for $\Q$VHS), which is a reductive $\Q$-group and the MT of each specialization $\VV_x$, $x\in X(\C)$ is naturally a subgroup of the generic one. So, as usual, we set

\begin{equation}
\HL (X,\VV^{\otimes}):=\{x\in X : \MT(\VV_x)\neq \MT(\VV)\}.
\end{equation}

\begin{question}
    What kind of structure does $ \HL (X,\VV^{\otimes}) \subset X(\C)$ have? E.g. is it a countable union of algebraic subvarieties (as in the usual setting)?
\end{question}

Employing the non-integral VHS constructed in \Cref{sec:intro}, in this final section,  we give examples where the natural analog of the Cattani-Deligne-Kaplan theorem \cite{CDK} fails, as well as the Andr\'{e}--Oort conjecture. Namely that there are examples of non-integral $\Q$VHS $(X,\VV)$ that have discrete monodromy, but however:
\begin{itemize}
\item $ \HL (X,\VV^{\otimes})$ is not a union of algebraic subvarieties of $X$ (cf. \Cref{thmnocdk});
\item $(X,\VV)$ contains a dense set of CM points, but $\VV$ is not of Shimura type (cf. \Cref{propnoAO}).
\end{itemize}
Finally, we propose a remedy by introducing the \emph{algebraic Hodge locus}, cf. \Cref{algebhl}, and work out an example.

\subsection{Proofs and further questions}
The first result shows that, in contrast with the famous Cattani-Deligne-Kaplan theorem \cite{CDK}, the Hodge locus is not algebraic:

\begin{theor}\label{thmnocdk}
There are infinitely many curves $X$ with the following property. The Hodge locus of $(X\ \times X, \VV \times \VV)$ cannot be written as a countable union of algebraic (nor analytic) subvarieties of $X\times X$.
\end{theor}
\begin{proof}
Let $X$ be a curve over $\mb{C}$ supporting $\VV$ a non-integral $\Q$VHS as constructed  in \Cref{thm: main} or \Cref{thm1.6}, which is uniformized by a map $\mb{H}\to X(\mb{C})$. We denote by $\Gamma=\pi_1(X) \subset \SL_2(\R)^+$ its  fundamental group. We consider the Hodge locus of  $(X\ \times X, \VV \times \VV)$. Let $\G$ be the Mumford-Tate group of $\VV$ (which, by construction, has the property that $\SL_2 \subset \G_\R$).

The non-integrality condition forces $\Gamma$ to be non-arithmetic. The commensurability criterion of Margulis \cite{margulisbook} then guarantees that there exists an element $g\in \G(\Q)_+ $ which does \textbf{not} commensurate $\Gamma$. In particular we can use such $g$ to find a totally geodesic embedding of $\mathbb{H}_g = (\id\times g)(\mb{H}) \subset \mathbb{H}\times \mathbb{H}$ (where $\mathbb{H} \subset \C$ denotes the upper half plane) such that its image in $(X\times X)(\mb C)$ is not closed, but nevertheless it gives a maximal subset of the Hodge locus. Such subset is not algebraic nor analytic, as desired.
\end{proof}

The phenomenon we have witnessed during the proof of \Cref{thmnocdk} is that the Hodge locus is related to Mumford-Tate subdata of $(\G\times \G,D\times S)$, without necessarily having a Zariski dense intersection with $\pi_1(X)$. This is related to the discussion regarding \emph{Real Shimura data} appearing  in \cite[Sec. 3]{BU}. The next proposition shows that the analogue of the André--Oort conjecture for $\Q$VHS that are not integral necessarily fail. 
\begin{prop}\label{propnoAO}
Let $X$ be a curve supporting $\VV$ a non-integral $\Q$VHS as constructed  in \Cref{thm: main} or \Cref{thm1.6}. The set of points $x\in X(\C)$ such that $\VV_x$ has CM (i.e. the Mumford--Tate group of the Hodge structure $\VV_x$ is commutative) is analytically dense in $X(\C)$. Furthermore the Hodge locus of $(X,\VV)$ is just made of CM points.
\end{prop}

\begin{proof}
The proof is almost identical to \cite[Prop. 7.1.2.]{BU}. In \emph{op. cit.} the authors find an example of a variety which is Shimura with a dense set of $K$-CM points for some totally real number field $K$ necessarily different from $\Q$.

Notice that it is enough to find one CM point and then act on it via the $\Q$-points of the generic Mumford-Tate group of $\VV$. One must exist because we can find a suitable maximal torus (essentially as in the case of Shimura varieties), cf. also the proof in \emph{op. cit.}. 
\end{proof}

In the example discussed above, when $(X,\VV)$ has a dense set of CM points, the conclusion is not that $X$ is a Shimura variety but the weaker one that $X$ is locally symmetric (or that its universal covering is a Hermitian symmetric space). It might be that a similar conclusion is always true:

\begin{question}
 A non-isotrivial $\Q$VHS on a smooth quasi-projective variety $(X,\VV)$ has a dense set of CM points if and only if $\VV$ is uniformizing (in particular the associated period domain is Hermitian).
\end{question}

\begin{question}
Are there examples of $\Q$VHS with discrete monodromy that are not uniformizing?
\end{question}

The above discussion suggests introducing a variant of the Hodge locus:
\begin{defi}\label{algebhl}
Let $(X,\VV)$ be a $\Q$VHS. The  \emph{algebraic Hodge locus} is the following subset of $X(\C)$:
\begin{equation}
\HL (X,\VV^{\otimes})_{\operatorname{alg}}= \bigcup _{Y \subset X  \ : \ \MT(\VV_{| Y^{nor}} ) \neq \MT(\VV)}  Y
\end{equation}
where the union is ranging over the algebraic subvarieties whose generic Mumford-Tate group ``drops'' (possibly after taking the normalization of the subvariety).
\end{defi}
The above variant  changes the picture only for positive dimensional subvarieties, and indeed points remain elusive (if we weren't to impose the non-isotriviality condition the above union would include every point $x\in X(\C)$). The above locus is very much related to the so called \emph{monodromy locus} (of positive period dimension), i.e. the locus of subvarieties  where the algebraic monodromy drops (non-trivially). This is studied in details in \cite{BKU, 2024arXiv240616628B} and, even more recently, Klingler has announced general results about the `rigid' part of the monodromy locus of arbitrary local systems; in contrast our investigation is focused on the case of $\Q$VHS.

We conclude with the description of the algebraic Hodge locus in the concrete family of examples studied  in \Cref{propnoAO}.
\begin{prop}
Let $X$ be a curve as in \Cref{propnoAO}. The algebraic Hodge locus of $(X\ \times X, \VV \times \VV)$ consists of:
\begin{itemize}
\item CM points,
\item vertical and horizontal lines at CM points,
\item finitely many modular correspondences parametrized by the commensurator of $\Gamma = \pi_1(X)\subset \mathbf{G}(\Q)$.
\end{itemize}
\end{prop}
\begin{proof}
Associated to $(X,\VV)$ we have a Shimura datum $(\mathbf{G},X_G)$. and we can consider sub-Shimura data of $(\mathbf{G}\times \mathbf{G} ,X_G\times X_G)$. Exactly as one does in the theory of Shimura varieties, we can also act on them via elements $g\in \mathbf{G}(\Q)_+$ (the so called \emph{Hecke action}). 

Consider: 
\begin{displaymath}
\HL(g\mathbf{G}g^{-1}):=\{(x,y)\in X\times X : \MT(\VV_x\times \VV_y)=g\mathbf{G}g^{-1} \}.
\end{displaymath}
It is the projection to $X\times X$ of a totally geodesic copy of $\mathbb{H}$ in $\mathbb{H}$ embedded via a twisted diagonal by $g$ (in particular compatible with the rational structures involved).
Observe that it is closed in $X\times X$ if and only if $g\in \Comm (\Gamma)$, the commensurator of $\Gamma$. Since $\Gamma$ is non-arithmetic, $\Comm (\Gamma) / \Gamma$ is finite, as desired.
\end{proof}

\begin{rmk}
We remark here that the approach to study the \emph{geometric part of the Hodge locus} appearing in \cite{2024arXiv240616628B} is quite general and various results hold true for the monodromy locus of $\Q$VHSs with discrete monodromy (but not the whole Hodge locus, as we have explained above)--more about this will appear in a forthcoming paper.
\end{rmk}

\bibliographystyle{abbrv}

\bibliography{biblio.bib}

\Addresses

\end{document}